\documentclass[11pt,a4paper]{amsart}

\usepackage{amsmath}
\usepackage{amsfonts}
\usepackage{amssymb}
\usepackage{hyperref}
\usepackage{amsthm}
\usepackage{graphicx}
\usepackage{dsfont,fullpage}

\newcommand{\into}{{\hookrightarrow}}
\newcommand{\R}{{\mathbb{R}}}
\newcommand{\C}{{\mathbb{C}}}

\newcommand{\spinc}{{\operatorname{spin}^c}}
\newcommand{\Cl}{{\C \text{l}}}
\newcommand{\E}{{\mathcal{E}}}
\newcommand{\ev}{{\operatorname{ev}}}

\newcommand{\1}{{\mathds{1}}}
\renewcommand{\d}{{\operatorname{d}}}

\def\O{\mathcal{O}}
\newcommand{\Z}{{\mathbb{Z}}}
\newcommand{\algotimes}{{\otimes_{\text{alg}}}}
\renewcommand{\L}{\mathcal{L}}

\def\tilde{\widetilde}
\def\epsilon{\varepsilon}
\def\Dtimes{i_!^\epsilon \times_{\nabla^\E} D_Y}

\DeclareMathOperator{\dom}{dom}

\DeclareMathOperator{\Index}{Index}
\DeclareMathOperator{\End}{End}
\DeclareMathOperator{\II}{II}
\DeclareMathOperator{\Tr}{Tr}
\DeclareMathOperator{\Lip}{Lip}

\newtheorem{proposition}{Proposition}[section]
\newtheorem*{proposition*}{Proposition}
\newtheorem{corollary}[proposition]{Corollary}
\newtheorem{lemma}[proposition]{Lemma}
\newtheorem{definition}[proposition]{Definition}
\newtheorem{remark}[proposition]{Remark}

\author{Walter D. van Suijlekom}
\address{Institute for Mathematics, Astrophysics and Particle Physics, Radboud University Nijmegen, Heyendaalseweg 135, 6525 AJ Nijmegen, The Netherlands}
\email{waltervs@math.ru.nl}

\author{Luuk S. Verhoeven}
\address{Department of Mathematics, University of Western Ontario, Middlesex College, N6A 5B7 London ON, Canada}
\email{lverhoe@uwo.ca}
\date{\today}  

\title{Riemannian embeddings in codimension one as unbounded $KK$-cycles}
\begin{document}

\begin{abstract}
	Given a codimension one Riemannian embedding of Riemannian spin$^c$-manifolds $\imath:X \to Y$ we construct a family $\{\imath_!^ \epsilon\}_{0< \epsilon< \epsilon_0}$ of unbounded $KK$-cycles from $C(X)$ to $C_0(Y)$, each equipped with a connection $\nabla^\epsilon$ and each representing the shriek class $\imath_! \in KK(C(X), C_0(Y))$.
	We compute the unbounded product of $\imath_!^\epsilon$ with the Dirac operator $D_Y$ on $Y$ and show that this represents the $KK$-theoretic factorization of the fundamental class $[X] = \imath_! \otimes [Y]$ for all $\epsilon$.
	In the limit $\epsilon \to 0$ the product operator admits an asymptotic expansion of the form $\frac 1 \epsilon T + D_X + \O(\epsilon)$ where the ``divergent'' part $T$ is an index cycle representing the unit in $KK(\C, \C)$ and the constant ``renormalized'' term is the Dirac operator $D_X$ on $X$. 
	The curvature of $(\imath_!^\epsilon, \nabla^\epsilon)$ is further shown to converge to the square of the mean curvature of $\imath$ as $\epsilon \to 0$.
\end{abstract}

\maketitle

\section{Introduction}
\label{sec:intro}

	In noncommutative geometry the central objects are spectral triples $(A, H, D)$ consisting of a $*$-algebra $A$, a Hilbert space $H$ on which $A$ is represented and a self-adjoint operator $D$ satisfying several axioms.
	One of the primary sources of examples of spectral triples are smooth $\spinc$ manifolds $X$ where $A = C^\infty(X)$, $H = L^2(\Sigma_X)$ for $\Sigma_X$ the spinor bundle over $X$, and $D = D_X$ is the Dirac operator associated to the $\spinc$ structure.
	Moreover, given a spectral triple where $A$ is commutative and several additional requirements hold, it is the spectral triple of a smooth $\spinc$ manifold and it is possible to recover the manifold $X$ from the abstract triple $(A, H, D)$ \cite{C96,C08}.

	A natural next question is what the maps in noncommutative geometry should be. One way to attack this problem is to see how smooth maps between manifolds can be encoded into the setting of noncommutative geometry.
	Because spectral triples can be interpreted as unbounded representatives for classes in $KK(A, \C)$, inspiration for this can be taken from $KK$-theory. 

	The shriek class of \cite{ConnesSkandalis_LongitudinalIndex} is particularly relevant.
	For the shriek class one starts from a smooth $K$-oriented map $f:X \to Y$ and associates a class $f_! \in KK(C(X), C(Y))$, in a way that is (contravariantly) functorial with respect to the Kasparov product, i.e. if $f:X \to Y$ and $g:Y \to Z$
	$$
		(g \circ f)_! = f_! \otimes g_! \in KK(C(X), C(Z)).
	$$
	
	Moreover, the spectral triple $(C^\infty(X), L^2(\Sigma_X), D_X)$ is an unbounded representative for the shriek class of the point map $\text{pt}:X \to \{*\}$.
	Thus, if we denote the class in $KK(C(X), \C)$ represented by the spectral triple of a $\spinc$ manifold $X$ by $[X]$, we have that
	\begin{equation}
		\label{eq:factorization}
		[X] = f_! \otimes [Y] \in KK(C(X), \C)
	\end{equation}
	for a smooth $K$-oriented map $f:X \to Y$.

	The factorization of Equation \ref{eq:factorization} takes place in $KK$-theory and is therefore topological in nature.
	However, we have spectral triples as unbounded and geometric representatives for $[X]$ and $[Y]$ while the construction of $f_!$ naturally lends itself to finding an unbounded representative \cite{BaajJulg} as well.
	This allows us to consider to what extent the factorization continues to hold at the unbounded level, before passing to $KK$-classes, and with an unbounded product instead as in \cite{Mesland_Correspondences,KaadLesch_SpectralFlow} instead of the Kasparov product. In this way we keep all geometric structure intact, including some notion of curvature that one may associate to the map $f$. 
	
	For instance, in the case where $f:X \to Y$ is a submersion, the unbounded cycle representing $f_!$ consists of a ``vertical spinor bundle'' $\mathcal{E}$ such that $L^2(\Sigma_X) \cong \mathcal{E} \otimes_{C(Y)} L^2(\Sigma_Y)$, equipped with a ``vertical'' family of Dirac operators $D_V$ and a connection $\nabla^\E$ \cite{KaadSuijlekom_Submersions}.
	The unbounded product, denoted by $D_V \times_{\nabla^\E} D_Y$, then produces the Dirac operator of $X$ up to a bounded curvature term:
	$$
		D_X = D_V \otimes 1 + 1 \otimes_{\nabla^\E} D_Y + \kappa,
	$$
	where $\kappa$ is proportional to the curvature of the submersion $f$.
	After passing down to $KK$-theory this gives Equation \ref{eq:factorization}.

	The other case in which the shriek class is tractable is when $f:X \to Y$ is an immersion. Previous work of us \cite{vanSuijlekomVerhoeven_AllSpheres} has dealt with the specific case of spheres embedded in Euclidean space, but here we treat the general codimension one case.
	Our main result is the construction of a family of unbounded representatives $\{\imath_!^\epsilon \}_{\epsilon}$ for the shriek class $\imath_!$ of a codimension one embedding $\imath:X \to Y$.
	In the limit $\epsilon \to 0$ we show that the spectral triple representing $[X]$ can be recovered from the unbounded products $\imath_!^\epsilon \times_{\nabla^\epsilon} D_Y$ as the constant term in a suitable asymptotic expansion around $\epsilon=0$.

	An important aspect of the unbounded product is the appearance of an unbounded cycle representing the multiplicative unit in $KK$-theory, which appears in the leading term $1/\epsilon$ of the asymptotic expansion in $\epsilon$. In analogy with renormalization methods
        in quantum field theory we may subtract this ``divergent'' term, take the limit $\epsilon \to 0$ and arrive at the ``renormalized'' term $D_X$, which for us is indeed the relevant contribution. 
	We further compute the curvature of $\imath_!^\epsilon$ in the sense of \cite{MeslandRennieSuijlekom_Curvature} and show that in the $\epsilon \to 0$ limit this gives us the mean curvature of the immersion $\imath$.

        \medskip 
	This article is organized as follows.	
	In section \ref{sec:geometry} we will cover the geometric setting as well as some useful geometric tools we will use throughout, like the second fundamental form and Fermi coordinates.
	Next, in section \ref{sec:immersion_module}, we will introduce our family of unbounded $KK$-cycles that represent the immersion.
	In this section we will also compute the product of an arbitrary element of this family with the spectral triple representing the ambient manifold, as well as compute the curvature of these unbounded cycles \`a la \cite{MeslandRennieSuijlekom_Curvature}.
	
	In section \ref{sec:analysis} we will cover the analytical aspects of this product using a result of \cite{LeschMesland_SelfadjointSums} on products of unbounded $KK$-cycles.
	Finally, in section \ref{sec:recovering}, we will discuss how this family of unbounded $KK$-cycles and their products allows us to recover the embedded manifold using an asymptotic expansion.
	Here we will also show how our construction is a refinement of the bounded construction using a notion of unbounded homotopy from \cite{vandenDungenMesland_Homotopy}.

        \subsection*{Acknowledgements}
        The authors would like to thank Alain Connes, Koen van den Dungen, Nigel Higson, Jens Kaad, Bram Mesland, Adam Rennie and George Skandalis for fruitful discussions and useful comments. 
        
\section{Geometric preliminaries}
\label{sec:geometry}

	In this section we will discuss the geometric setting of our construction.
	We will start with a brief introduction of Fermi coordinates and some associated constructions, followed by a series of relations between the Levi--Civita connections on the embedded and ambient manifolds and the corresponding relation between their Dirac operators.
	
	Throughout this article, let $X$ be a compact $2k$-dimensional smooth Riemannian manifold, $Y$ a $2k+1$-dimensional smooth Riemannian $\spinc$ manifold and $\imath:X \into Y$ a smooth, oriented Riemannian embedding.
	We will denote the unit normal vector field to $\imath$ by $\nu$.

	Define $\tilde{\imath}:X \times \R \to Y$ by $\tilde{\imath}(x, s) = \exp_{\imath(x)}(s\nu)$, where $\exp$ is the exponential map of $Y$.
	It is well-known, see for example \cite[Lemma 2.3]{Gray_Tubes}, that there exists some $\epsilon_0 > 0$ such that $\tilde{\imath}$ becomes a diffeomorphism onto its range if the domain is restricted to $X \times (-\epsilon_0, \epsilon_0)$.
	This allows us to define the ``normal'' coordinate $\textbf{s}$ in a neighborhood $U$ of $X \subset Y$.
	If $(V, \psi)$ is a coordinate patch for $X$ it can be extended to a coordinate patch for $Y$ as $( \tilde{\imath}(V \times (-\epsilon_0, \epsilon_0)), (\psi \circ \tilde{\imath}^{-1}, \textbf{s}))$, such coordinate patches are called Fermi coordinates.
	
	Moreover, this normal coordinate function $\textbf{s}$ allows us to foliate the neighbourhood of $X \subset Y$ by leaves $X_s := \tilde{\imath}(X, s) = \{y \in U \subset Y | \textbf{s}(y) = s\}$ diffeomorphic to $X$.
	Each $X_s$ is a Riemannian manifold with metric induced from the metric on $Y$.
	By the generalized Gauss Lemma (see e.g. \cite[Corollary 2.14]{Gray_Tubes}), the vector field $\partial_s = \frac{\partial}{\partial \textbf{s}}$ is the unit normal vector field to each $X_s$ and $\partial_s|_{X_0} = \nu$.
	This constructions gives us a family of Riemannian manifolds, $\{X_s\}_{s \in (-\epsilon_0, \epsilon_0)}$, each of which is canonically diffeomorphic to $X$ by the map $x \mapsto \exp_{\imath(x)}(s\nu)$ but this diffeomorphism is in general not Riemannian, except from $X$ to $X_0$ where it coincides with $\imath$.

	\begin{figure}
		\centering
		\includegraphics[width=0.8\textwidth]{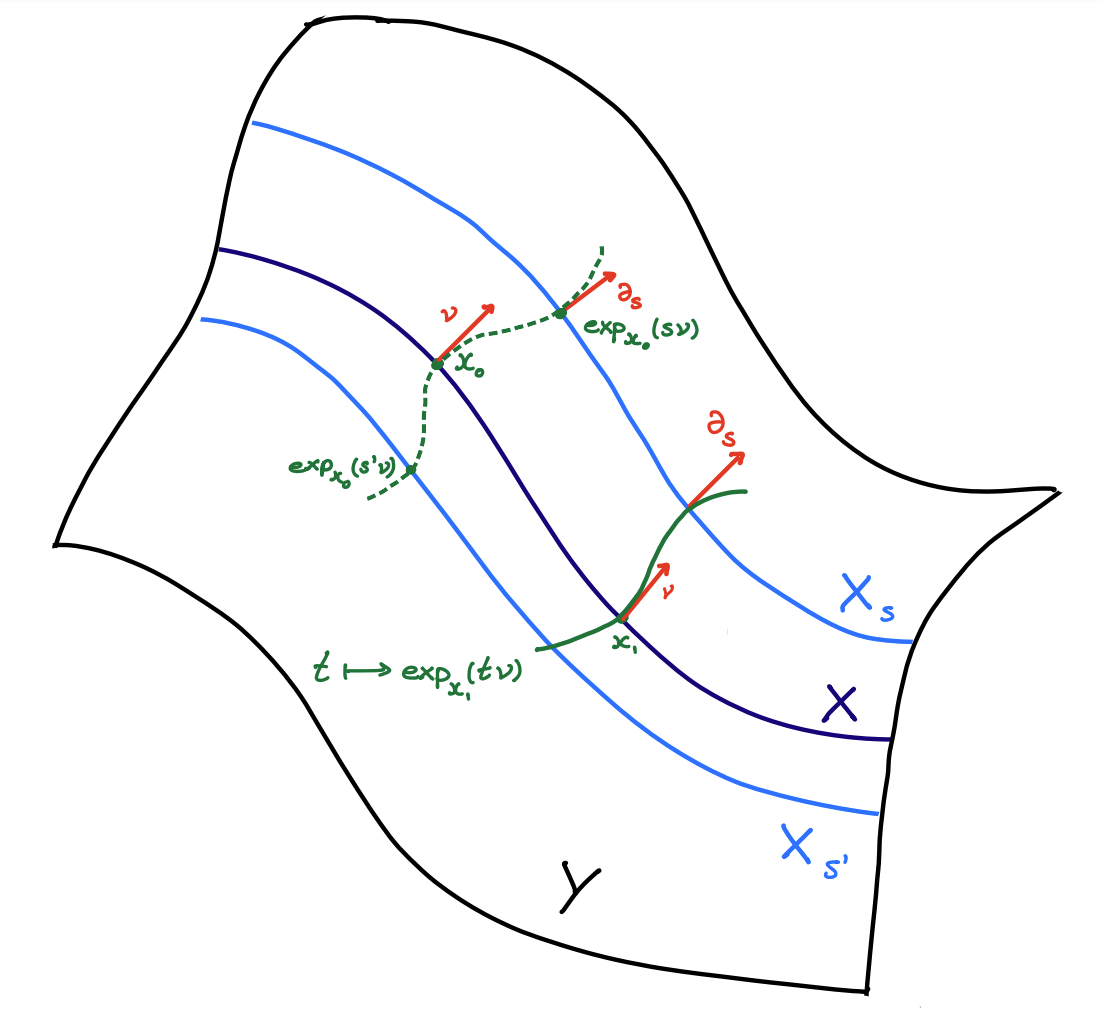}
		\caption{A diagram showing a patch of the manifold $Y$ and how the family $\{X_s\}_{s \in (-\epsilon_0, \epsilon_0)}$ lies inside $Y$. The two green curves show two paths on $Y$ defined by geodesic flow along $\nu$, i.e. $t \mapsto \exp_{x}(t\nu)$, starting from $x_0 \in X$ (dashed) and $x_1 \in X$ (solid).
		For some $s > t$ and $s' < t$ the submanifolds $X_s$ and $X_{s'}$ are marked in light blue. These submanifolds consist of all points that are reached by green curves at $t = s$ and $t = s'$ respectively. While all $X_s$ are diffeomorphic by the geodesic (green) curves, the metric along $X_s$ can vary.
		The normal vector field $\partial_s$ is defined by the tangent vector to the geodesic (green) curves.
		Points along a geodesic (green) curve can be uniquely identified by their starting point in $X$ and distance $s$ along the geodesic, giving the Fermi coordinates.}
		\label{fig:Fermi_coords}
	\end{figure}

	Fermi coordinates, or more specifically the existence of a global normal coordinate, allows us to simplify vector bundles over $\tilde{\imath}(X \times (-\epsilon_0, \epsilon_0))$, as follows.
	
	\begin{lemma}
		\label{lem:bundles_are_simple}
		Let $E$ be a vector bundle over $Y$ equipped with a metric connection $\nabla^E$.
		Then for any $s, s' \in (-\epsilon_0, \epsilon_0)$ there is an isomorphism of vector bundles $U_{s',s}:E|_{X_s} \to E|_{X_{s'}}$.
		Moreover, $U_{s'',s'}U_{s',s} = U_{s'',s}$.
	\end{lemma}
	\begin{proof}
		For $v \in E_{\tilde{\imath}(x, s)}$ define $U_{s',s,x}v \in E_{\tilde{\imath}(x, s')}$ by parallel transport along the curve $\exp_{\imath(x)}(t\nu)$ for $t$ between $s$ and $s'$ for the connection $\nabla^E$.
		Each $U_{x,s',s}$ is a linear isomorphism, and they assemble into an isomorphism of vector bundles $U_{s',s}:E|_{X_s} \to E|_{X_{s'}}$.
		The composition property follows from the composition of parallel transport along a chosen curve.
	\end{proof}
	\begin{corollary}
		\label{cor:section_extension_property}
		Let $E$ be as above and $s$ fixed.
		Then any section of $E|_{X_s}$ can be extended to a section of $E$.
	\end{corollary}
	\begin{proof}
		Let $b:\R \to \R$ be a smooth function satisfying $b(s) = 1$ and $b \equiv 0$ outside of $(-\epsilon_0, \epsilon_0)$.
		If $\psi$ is a section of $E|_{X_s}$ define $\tilde{\psi} \in \Gamma(E)$ to be $0$ outside of $\tilde{\imath}(X \times (-\epsilon_0, \epsilon_0))$ and $\tilde{\psi}(\tilde{\imath}(x,t)) = b(t)U_{t,s}\psi(x)$.
	\end{proof}

	\begin{remark}
		The most common usage of the above Lemma and Corollary is to construct Fermi frames, these are extensions of frames over some open subset of $X$ to a frame over some open subset in $Y$.
		This is accomplished by extending the frame using the Corollary and adding to it the normal vector field $\partial_s$.
		Due to the definition of parallel transport these frames have the desirable property that $\nabla^Y_{\partial_s} = \frac{\partial}{\partial s}$.
	\end{remark}

	In essence we are interested in comparing two different metrics on the manifold $X \times (-\epsilon_0, \epsilon_0)$.
	On the one hand we have the product metric, $\pi^* g_X$, where $\pi:X \times (-\epsilon_0, \epsilon_0) \to X$ is the projection on the first coordinate.
	On the other hand we have the metric induced by $\tilde{\imath}$ and $Y$, $\tilde{\imath}^* g_Y$.
	These agree along $X \times \{0\}$, but will in general differ everywhere else.

	In this vein we introduce the function $\Lambda$, it measures the change in volume between $X$ and $X_s$ and is defined by
	\begin{equation}
		\label{eq:definition_Lambda}
		\Lambda \cdot \tilde{\imath}^* \omega_Y = \pi^* \omega_X.
	\end{equation}
	In local Fermi coordinates $\Lambda(x, s) = \sqrt{\frac{\det g_Y(x, 0)}{\det g_Y(x, s)}}$.

	We will also use the second fundamental form, denoted $\II$, associated to the family of embeddings $x \mapsto \tilde{\imath}(x, s)$.
	Recall (see e.g. \cite{Ballman_Fundamental}) that $\II_{(x, s)}$ is a bilinear form on $T_{\tilde{\imath}(x, s)}X_s$ and relates the Levi--Civita connections of $Y$ and $X_s$ by
	\begin{align}
		& \nabla^Y_A(B) = \nabla^{X_s}_A(B) + \II_{(x, s)}(A, B)\partial_s, \label{eq:second_fundamental} \\
		& \II_{(x, s)}(A, B) := \langle \nabla^{Y}_A(B), \partial_s \rangle = -\langle B,  \nabla^Y_A(\partial_s) \rangle, \nonumber
	\end{align}
	for $A, B \in T_{\tilde{\imath}(x, s)}X_s$.
	The second fundamental form and our scaling function $\Lambda$ turn out to be related:
	\begin{lemma}
		With the notation introduced in the preceding paragraphs,
		$$
			\frac{1}{\Lambda}\left[ \nabla^Y_{\partial_s}, \Lambda \right] = \Tr(\II).
		$$
		\label{lem:Tr_II_as_derivative_Lambda}
	\end{lemma}
	\begin{proof}
		Since the statement is coordinate invariant, we can verify this in a Fermi frame where $\nabla^Y_{\partial_s} = \frac{\partial}{\partial s}$. 
		The result then follows from \cite[Theorem 3.11]{Gray_Tubes}.
	\end{proof}

	We will now use Corollary \ref{cor:section_extension_property} to establish that each $X_s$ is a $\spinc$ manifold with a $\spinc$ structure inherited from $Y$ and the embedding $\imath$.
	Following this we will do a series of computations to relate the Dirac operators on each of these $X_s$ to the Dirac operator on the ambient space $Y$.
	Throughout this article we will maintain the convention that Clifford multiplication by a vector is self-adjoint, in particular $\gamma_\mu \gamma_\nu + \gamma_\nu \gamma_\mu = +2g_{\mu\nu}$.

	\begin{lemma}
		\label{lem:spinc_on_Xs}
		Let $\Sigma_Y$ be a spinor bundle over $Y$.
		Then the vector bundle $\Sigma_{X_s} = \Sigma_Y |_{X_s}$ is a spinor bundle over $X_s$ with Clifford multiplication $c_s(A) = ic_Y(A)c_Y(\partial_s)$ for $A \in T X_s$.
	\end{lemma}
	\begin{proof}
		Because $T_{\tilde{\imath}(x, s)}Y = T_{\tilde{\imath}(x,s)}(X_s) \oplus \R (\partial_s)_{\tilde{\imath}(x, s)}$ for all $(x, s) \in X \times (-\epsilon_0, \epsilon_0)$, we get an isomorphism of Clifford bundles $\Cl(TX_s) \cong \Cl^0(TY)|_{X_s}$.
		
		Since $Y$ is odd-dimensional, $\Cl^0(TY) \cong \End(\Sigma_Y)$, hence it remains to be shown that we also have $\End(\Sigma_Y)|_{X_s} \cong \End(\Sigma_Y|_{X_s})$.
		The inclusion $\subseteq$ is clear, to obtain the other inclusion we apply Corollary \ref{cor:section_extension_property} to the vector bundle $\End(\Sigma_Y) \cong \Sigma_Y \otimes \Sigma_Y^*$ which is equipped with a metric connection induced from the metric Clifford connection on $\Sigma_Y$.
	\end{proof}
	
	\begin{lemma}
		Let $\nabla^{\Sigma_Y}$ denote the metric Clifford connection on $\Sigma_Y$, $A \in \mathfrak{X}(Y)$ a tangential vector field, i.e. $\langle A, \partial_s \rangle \equiv 0$, and $\psi \in \Gamma^\infty(\Sigma_Y)$.
		Then
		$$
			\nabla^{\Sigma_{X_s}}_{A|_{X_s}}\left(\psi|_{X_s}\right) = \left(\nabla^{\Sigma_Y}_A \psi\right)|_{X_s} + \frac{1}{2}ic_s\left(\nabla^Y_A(\partial_s)|_{X_s}\right)\left(\psi|_{X_s}\right)
		$$
		defines a metric Clifford connection on $\Sigma_{X_s}$.
	\end{lemma}
	\begin{proof}
		By Corollary \ref{cor:section_extension_property} the given formula indeed defines the connection for all vector fields on $X_s$ and sections of $\Sigma_{X_s}$.
		Because the entire expression is $C_0(Y)$-linear in $A$ it immediately follows that $\nabla^{\Sigma_{X_s}}$ only depends on $A|_{X_s}$, so that the definition of $\nabla^{\Sigma_{X_s}}$ does not depend on the specifics of the extension procedure.
		The dependence on $\psi$ is slightly more subtle since the term $\nabla^Y_A(\psi)$ is not $C_0(Y)$-linear in $\psi$.
		However, since $A$ is tangential it is $C(-\epsilon_0, \epsilon_0)$-linear where $f \in C(-\epsilon_0, \epsilon_0)$ acts by $(f\psi)(x, s) = f(s)\psi(x, s)$, showing that $\nabla^{\Sigma_{X_s}}$ only depends on $\psi|_{X_s}$.

		We also need to check that $c_s\left(\nabla^Y_A(\partial_s)|_{X_s}\right)$ is well defined by showing that $\nabla_A^Y(\partial_s)$ is indeed tangential.
		This follows quickly from the fact that $\partial_s$ has unit norm,
		\begin{align*}
			\langle \nabla^Y_A(\partial_s), \partial_s \rangle & = A(\langle \partial_s, \partial_s \rangle) - \langle \partial_s, \nabla^Y_A(\partial_s) \rangle, \\
			\langle \nabla^Y_A(\partial_s), \partial_s \rangle & = - \langle \nabla^Y_A(\partial_s), \partial_s \rangle.
		\end{align*}

		Now that we have established that $\nabla^{\Sigma_{X_s}}$ is well-defined, let us check that it is metric and Clifford.
		In the following computations, let $A, B \in \mathfrak{X}(X_s)$.
		\begin{align*}
			\langle \nabla^{X_s}_A \psi, \phi \rangle + \langle \psi, \nabla^{X_s}_A \phi \rangle & = \langle \nabla^Y_A \psi, \phi \rangle + \langle \psi, \nabla^Y_A \phi \rangle + \langle \frac{1}{2} i c_s(\nabla^Y_A(\partial_s))\psi, \phi \rangle + \langle \psi, \frac{1}{2} i c_s(\nabla^Y_A(\partial_s)) \phi \rangle, \\
			& = A(\langle \psi, \phi\rangle) - \langle \psi, \frac{1}{2}ic_s(\nabla^Y_A(\partial_s))\phi \rangle + \langle \psi, \frac{1}{2} i c_s(\nabla^Y_A(\partial_s)) \phi \rangle, \\
			& = A(\langle \psi, \phi \rangle),
		\end{align*}
		where we used our convention that Clifford multiplication by a vector is self-adjoint.

		Now we check the Clifford compatibility.
		Using the definition of the second fundamental form, equation \ref{eq:second_fundamental}, we can compute
		\begin{align*}
			\left[\nabla^{\Sigma_{X_s}}_A, c_s(B)\right] 
			& = \left[\nabla^{\Sigma_Y}_A, ic_Y(B)c_Y(\partial_s)\right] + \frac{1}{2}i \left[ c_s(\nabla^{\Sigma_Y}_A(\partial_s)), c_s(B) \right], \\
			& = ic_Y\left(\nabla^Y_A(B)\right)c_Y(\partial_s) + ic_Y(B)c_Y\left(\nabla^Y_A(\partial_s)\right) + \frac{1}{2}i \left[ c_s(\nabla^Y_A(\partial_s)), c_s(B) \right], \\
			& = c_s\left(\nabla^{X_s}_A(B)\right) + ic_Y(\II(A, B)\partial_s)c_Y(\partial_s) \\ & \hspace{10em} + ic_s(B)c_s(\nabla^Y_A(\partial_s)) + \frac{1}{2}i \left[ c_s(\nabla^Y_A(\partial_s)), c_s(B) \right], \\
			& = c_s\left(\nabla^{X_s}_A(B)\right) - i\langle B, \nabla^Y_A(\partial_s) \rangle + i\frac{1}{2}\left\{c_s(B), c_s(\nabla^Y_A(\partial_s))\right\}, \\
			& = c_s\left(\nabla^{X_s}_A(B)\right).
		\end{align*}
	\end{proof}
	
	We have now established that each $X_s$ is a $\spinc$-manifold and can be equipped with a metric Clifford connection induced by the Clifford connection on $Y$.
	This means that, in particular, we can define Dirac operators $D_{X_s}$ for each $X_s$.
	These Dirac operators are related to the Dirac operator on $Y$ as follows.

	\begin{lemma}
		\label{lem:Dirac_expansion}
		On the neighborhood $\tilde{\imath}(X \times (-\epsilon_0, \epsilon_0))$ of $\imath(X) \subset Y$ the Dirac operators of $Y$ and $X_s$ are related by
		$$
			(D_Y \psi)(x, s) = ic_Y(\partial_s)\left( (D_{X_s} \psi|_{X_s})(x) - \frac{1}{2} \Tr(\II_s)\psi(x, s) + \nabla^{\Sigma_Y}_{\partial_s}(\psi)(x, s) \right).
		$$
	\end{lemma}
	\begin{proof}
		Let $\{e_1, ..., e_n, \partial_s\}$ be a local orthonormal frame, then with summation over the repeated indices $j, k$
		\begin{align*}
			(D_Y \psi)(x, s) & = ic_Y(e_j)\nabla^{\Sigma_Y}_{e_j}(\psi)(x, s) + ic_Y(\partial_s)\nabla^{\Sigma_Y}_{\partial_s}(\psi)(x, s), \\
			& = -c_Y(\partial_s)c_s(e_j)\left(\nabla^{\Sigma_{X_s}}_{e_j}(\psi)(x, s) - \frac{1}{2}ic_s\left(\nabla^{\Sigma_Y}_{e_j}(\partial_s)\right)\psi(x, s)\right) + ic_Y(\partial_s)\nabla^{\Sigma_Y}_{\partial_s}(\psi)(x, s), \\
			& = ic_Y(\partial_s)\left(ic_s(e_j)\nabla^{\Sigma_{X_s}}_{e_j}(\psi|_{X_s})(x) - \frac{1}{2} c_s(e_j) c_s(\II(e_j, e_k)e_k)\psi(x, s) + \nabla^{\Sigma_Y}_{\partial_s}(\psi)(x, s)\right), \\
			& = ic_Y(\partial_s)\left((D_{X_s}\psi|_{X_s})(x) - \frac{1}{2} \Tr(\II)\psi(x, s) + \nabla^{\Sigma_Y}_{\partial_s}(\psi)(x, s)\right)
		\end{align*}
		where we get the trace because $\II$ is symmetric.
	\end{proof}
	\begin{remark}
		For future use we will write $D_{X_\bullet}$ for the operator defined on (a suitable domain in) $L^2(X \times (-\epsilon_0, \epsilon_0))$ by $(D_{X_\bullet} \psi)(x, s) = (D_{X_s}\psi|_{X_s})(x)$.
		In this notation Lemma \ref{lem:Dirac_expansion} can be written more succinctly $D_Y = ic_Y(\partial_s)(D_{X_\bullet} - \frac{1}{2}\Tr(\II) + \nabla^{\Sigma_Y}_{\partial_s})$.
	\end{remark}
	
	To conclude this section we will describe the unbounded $KK$-cycles, or more precisely \emph{spectral triples}, associated to the Riemannian $\spinc$ manifolds $X$ and $Y$.
	Since $X$ is even-dimensional, we associate to it the graded spectral triple $(C(X), L^2(\Sigma_{X_0}), D_{X_0}; c_0(\nu))$ which represents a $KK$-cycle between $C(X)$ and $\C$; we will also simply write $D_X$ to denote this spectral triple. 
	To the odd-dimensional $Y$ we associate an unbounded $KK$-cycle between $C_0(Y) \otimes \Cl_1$ and $\C$ given by $(C_0(Y) \otimes \Cl_1, L^2(\Sigma_Y) \otimes \C^2, D_Y \otimes \sigma_2; 1 \otimes \sigma_3)$ where the left-action of the generator of $\Cl_1$ is given by $\sigma_1$. Again, we will use the short-hand notation $D_Y$ for this spectral triple. 
	
\section{Family of immersion modules}
\label{sec:immersion_module}

	In this section we will define the family of unbounded $KK$-cycles $\{\imath_!^\epsilon  \}_{\epsilon \in (0, \epsilon_0)}$, with $\epsilon_0$ as defined in section \ref{sec:geometry}.
	We will also equip each cycle in this family with a connection $\nabla^\epsilon$ and compute the unbounded product $\imath_!^\epsilon \times_{\nabla^\epsilon} D_Y$ in the sense of \cite{KaadLesch_SpectralFlow} and \cite{Mesland_Correspondences}.
	The construction of this family is based on the shriek class appearing in \cite{ConnesSkandalis_LongitudinalIndex} and, more directly, generalizes previous work of the authors \cite{vanSuijlekomVerhoeven_AllSpheres}.
	In section \ref{sec:recovering} we will discuss how this family of $KK$-cycles represents the shriek class and how to recover the Dirac operator $D_X$ on the embedded manifold from the family of unbounded products $\imath_!^\epsilon \times_{\nabla^\epsilon} D_Y$.
	
	Fix an $\epsilon \in (0, \epsilon_0)$ and define $\E = C_0(X \times (-\epsilon, \epsilon), \Cl_1)$.
	We equip $\E$ with the structure of a $C(X)$-$C_0(Y) \otimes \Cl_1$ bimodule by
	$$
		(g \psi h)(x, s) = g(x)\psi(x, s)h(\tilde{\imath}(x, s))
	$$
	for $g \in C(X)$, $h \in C_0(Y) \otimes \Cl_1$, and $\psi \in \E$.
	The bimodule $\E$ inherits a grading from $\Cl_1$.
	To obtain a Hilbert bimodule we additionally need to define a $C_0(Y) \otimes \Cl_1$-valued inner product on $\E$.
	Set
	$$
		\langle \psi, \phi \rangle_\E (y) = \left\{ \begin{array}{c l}
			\Lambda(x, s) \overline{\psi(x, s)}\phi(x, s) & y = \tilde{\imath}(x, s),\, s \in (-\epsilon, \epsilon), \\
			0 & \text{elsewhere}.
		\end{array}	 \right.
	$$
	with $\Lambda$ as in \ref{eq:definition_Lambda}.
	
	We then obtain an unbounded $KK$-cycle by further equipping $\E$ with the operator $S$ defined by
	$$
		(S \psi)(x, s) = f(s) \sigma_1 \psi(x, s)
	$$
	for $f(s) = -\frac{\pi}{2\epsilon}\tan\left(\frac{\pi s}{2 \epsilon}\right)$ and domain $\dom(S) = \{ \psi \in \E | S \psi \in \E\}$.
	
	\begin{remark}
		Note that at this point there is a mismatch with \cite{vanSuijlekomVerhoeven_AllSpheres} where $f$ was defined as $+\frac{\pi}{2\epsilon}\tan\left(\frac{\pi s}{2\epsilon}\right)$.
		In order to obtain the correct orientation on all cycles the choice of the $-$ is the correct one.
	\end{remark}

	For convenience we will include the proof that $\imath_!^\epsilon = (\E, S)$ indeed defines an unbounded $KK$-cycle, as also found in \cite[Prop 2.4]{vanSuijlekomVerhoeven_AllSpheres}.
	Note that $\E$ and $S$ depend on $\epsilon$, this will be made explicit in section \ref{sec:recovering} where this dependence becomes the focus.
	For now we will hide this dependence to keep the notation lighter.

	\begin{proposition}
		The data $\imath_!^\epsilon = (\E, S)$ as defined above defines a $KK$-cycle.
	\end{proposition}
	\begin{proof}
		The norm induced on $\E$ by $\langle \cdot, \cdot \rangle_\E$ is $\|\psi\|_\E = \sup_{(x, s) \in X \times (-\epsilon, \epsilon)} | \Lambda(x, s) \psi(x, s) |$.
		As $0 < \epsilon < \epsilon_0$, the smooth function $\Lambda$ is bounded and strictly positive on all of $X \times (-\epsilon, \epsilon)$ \cite[Lemma 3.9]{Gray_Tubes} so that $\|\psi\|_\E$ is equivalent to the $\sup$-norm on $C_0(X \times (-\epsilon, \epsilon))$.
		Hence $\E$ is complete as a Hilbert module.

		The verification that $(\E, S)$ defines an unbounded $KK$-cycle then requires us to check that $S$ is self-adjoint, regular and has compact resolvent.
		This all follows from an investigation of
		$$
			(S \pm i)^{-1} = \frac{1}{1+f^2}(f \sigma_1 \mp i).
		$$
		Since $(S \pm i)^{-1}$ is multiplication by an element of $C_0(X \times (-\epsilon, \epsilon), \Cl_1)$ it is, as required, a compact operator on $\E$ and since $(S \pm i)^{-1}\E \subset \dom S$ we find that $S \pm i$ is surjective so that $S$ is self-adjoint and regular.
	\end{proof}
	
	We further want to equip $\imath_!^\epsilon$ with a connection.
	To define this connection we will make the identification $\E \subset C_0(Y) \otimes \Cl_1$, $\psi \mapsto \psi \circ \tilde{\imath}^{-1}$ extended by 0 outside the range of $\tilde{\imath}$.
	With this identification a universal connection is defined by 
	\begin{align*}
		\nabla^\E_u:\E & \rightarrow \E \otimes \Omega^1_u(C_0(Y) \otimes \Cl_1), \\
		\psi & \mapsto 1 \otimes \delta(\psi) + 1 \otimes \frac{1}{2\Lambda}\delta(\Lambda)\psi.
	\end{align*}
	The corresponding connection relative to $(L^2(\Sigma_Y) \otimes \C^2, D_Y \otimes \sigma_2)$ is
	\begin{align*}
		\nabla^\E_{D_Y \otimes \sigma_2}: \E^\infty & \rightarrow \E \otimes \Omega^1_{D_Y \otimes \sigma_2}(C_0(Y) \otimes \Cl_1), \\
		\psi & \mapsto 1 \otimes [D_Y \otimes \sigma_2, \psi] + \gamma(\psi) \otimes \frac{1}{2\Lambda}[D_Y \otimes \sigma_2, \Lambda]
	\end{align*}
	where $\E^\infty = C^\infty_0(Y) \otimes \Cl_1$, $\gamma$ is the grading operator on $\E$ and we use the graded commutativity of $[D_Y \otimes \sigma_2, \Lambda]$ and $\psi \in C^\infty_0(Y)\otimes \Cl_1$.

	To this connection we associate a curvature operator $(\nabla^\E)^2$ in the sense of \cite{MeslandRennieSuijlekom_Curvature}.
	However, the algebra $C_0(Y) \otimes \Cl_1$ is graded, by the grading of $\Cl_1$, so some of the spaces involved have to be adapted slightly.
	In particular we have, for $B$ a graded $*$-algebra with grading $\gamma$, $\Omega^1_u(B) = \ker\left(m_\gamma:B \otimes B \to B\right)$, $m_\gamma(a \otimes b) = a \gamma(b)$ and $\Omega^1_u(B)$ is itself graded by $-\gamma \otimes \gamma$.
	The universal differential $\delta:B \to \Omega^1_u(B)$ is given by $\delta(b) = 1 \otimes b - \gamma(b) \otimes 1$.

	\begin{lemma}
		Let $B$ be a graded $*$-algebra, then $\nabla_u:B \to B \otimes_B \Omega^1_u(B)$ given by
		$$
			\nabla_u(b) = 1 \otimes \delta(b) + 1 \otimes \omega b
		$$
		is a connection provided that $\gamma(\omega) = -\omega$ and the curvature relative to $(B, H, D)$ is given by
		$$
			\pi_D(\nabla_u^2)(b \otimes \psi) = 1 \otimes (m \circ (\pi_D \otimes \pi_D))(\delta(\omega) + \omega \otimes \omega)b\psi
		$$
		as map $B \otimes_B H \to B \otimes_B H$.
	\end{lemma}
	\begin{proof}
		The verification that $\nabla_u$ is a connection follows immediately from the graded Leibniz rule $\delta(bc) = \delta(b)c + \gamma(b)\delta(c)$, we also have that $\nabla_u$ is odd since both $\delta$ and $\omega$ are.
		We can then compute
		\begin{align*}
			(1 \otimes_{\nabla_u} \delta)(\nabla_u(b)) & = (1 \otimes_{\nabla_u} \delta)(1 \otimes \delta(b) + 1 \otimes \omega b), \\
			& = \gamma(1) \otimes \delta^2(b) + \nabla_u(1) \otimes \delta(b) + 1 \otimes \delta(\omega b) + \nabla_u(1) \otimes \omega b, \\
			& = 1 \otimes \omega \otimes \delta(b) + 1 \otimes (\delta(\omega)b + \gamma(\omega)\delta(b)) + 1 \otimes \omega \otimes \omega b, \\
			& = 1 \otimes \omega \otimes \delta(b) + 1 \otimes \delta(\omega)b - 1 \otimes \omega\delta(b)) + 1 \otimes \omega \otimes \omega b, \\
			& = 1 \otimes (\delta(\omega) + \omega \otimes \omega)b.
		\end{align*}		
	\end{proof}
	\begin{remark}
		This Lemma also follows from \cite[Proposition 3.4]{MeslandRennieSuijlekom_Curvature}, up to considerations due to the grading on $B$.
		We include this proof because it uses a different approach and highlights the appearances of the grading on $B$ and $\Omega^1_u(B)$.
	\end{remark}
	\begin{corollary}
		The curvature of $(\imath_!^\epsilon, \nabla^\E_u)$ relative to $(C_0(Y) \otimes \Cl_1, L^2(\Sigma_Y) \otimes \C^2, D_Y \otimes \sigma_2)$ is the operator
		$$
			\pi_{D_Y \otimes \sigma_2}\left( (\nabla^\E)^2 \right) = 1_{\E} \otimes \left( \frac{1}{2\Lambda}[D_Y, \Lambda]\right)^2 \otimes 1_{\C^2} = 1_{\E} \otimes \left(\frac{1}{4\Lambda^2}[D_{X_\bullet}, \Lambda]^2 - \frac{1}{4}\Tr(\II)^2\right) \otimes 1_{\C^2}
		$$
		acting on $\E \otimes_{C_0(Y) \otimes \Cl_1} (L^2(\Sigma_Y) \otimes \C^2)$.
	\end{corollary}
	\begin{proof}
		The previous Lemma applies, interpreting $\E^\infty \subset C^\infty_0(Y) \otimes \Cl_1$ via $\tilde{\imath}^{-1}$.
		We then have $\omega = \frac{1}{2\Lambda}\delta(\Lambda)$, which satisfies $\delta(\omega) = 0$.
		Moreover, since $D_Y$ is first-order $\Omega^1_{D_Y \otimes \sigma_2}(C^\infty_0(Y))$ and $C_0(Y) \otimes \Cl_1$ is graded commutative, we get the first equality.

		The second equality follows from the decomposition in Lemma \ref{lem:Dirac_expansion}, the anti-commutation between $[D_{X_s}, \Lambda]$ and $c_Y(\partial_s)$ to cancel the cross terms and Lemma \ref{lem:Tr_II_as_derivative_Lambda} to obtain the $\Tr(\II)$ term.		
	\end{proof}
	
	\begin{lemma}
		\label{lem:product_unitary}
		Using the identifications $(\Sigma_Y)_{(x, s)} \cong (\Sigma_Y)_{(x, 0)} = (\Sigma_X)_x$ from Lemma \ref{lem:bundles_are_simple}, the map
		\begin{align*}
			& U:\E^\infty \otimes_{C^\infty_0(Y) \otimes \Cl_1} (\Gamma^\infty(\Sigma_Y) \otimes \C^2) \to \Gamma^\infty(\pi^*\Sigma_X) \otimes \C^2, \\
			& U(\psi \otimes \phi)(x, s) = \psi(x, s)\phi(\tilde{\imath}(x, s))
		\end{align*}
		extends to a unitary map $\E \otimes_{C_0(Y) \otimes \Cl_1} (L^2(\Sigma_Y) \otimes \C^2) \to L^2(\pi^*\Sigma_X) \otimes \C^2$. Moreover, this map is compatible with the gradings provided $L^2(\pi^*\Sigma_X) \otimes \C^2$ is graded by $1 \otimes \sigma_3$.
	\end{lemma}
	\begin{proof}
		Let $\psi_1 \otimes \phi_1, \psi_2 \otimes \phi_2 \in \E \otimes_{C_0(Y) \otimes \Cl_1} (L^2(\Sigma_Y) \otimes \C^2)$, then
		\begin{align*}
			\langle \psi_1 \otimes \phi_1, & \psi_2 \otimes \phi_2 \rangle \\
			& = \langle \phi_1, \langle \psi_1, \psi_2 \rangle_\E \phi_2 \rangle, \\
			& = \int_Y \langle \phi_1(y), \Lambda(y)\overline{\psi_1(\tilde{\imath}^{-1}(y))}\psi_2(\tilde{\imath}^{-1}(y)) \phi_2(y) \rangle_{\Sigma_Y \otimes \C^2} \d y, \\
			& = \int_{X \times (-\epsilon, \epsilon)} \Lambda(x, s) \langle \psi_1(x, s)\phi_1(\tilde{\imath}(x, s)), \psi_2(x, s)\phi_2(\tilde{\imath}(x, s)) \rangle_{\Sigma_Y \otimes \C^2} \sqrt{\det g_Y(x, s)} \d x \d s, \\
			& = \int_{X \times (-\epsilon, \epsilon)} \langle \psi_1(x, s)\phi_1(\tilde{\imath}(x, s)), \psi_2(x, s)\phi_2(\tilde{\imath}(x, s)) \rangle_{\Sigma_X \otimes \C^2} \sqrt{\det g_Y(x, 0)} \d x \d s, \\
			& = \langle U(\psi_1 \otimes \phi_1), U(\psi_2 \otimes \phi_2) \rangle.
		\end{align*}
		
		To see that the corresponding grading on $L^2(\pi^*\Sigma_X) \otimes \C^2$ is $1 \otimes \sigma_3$, note that $L^2(\Sigma_Y) \otimes \C^2$ is graded by $1 \otimes \sigma_3$ and the grading of the $\Cl_1$-factor in $\E$ is merged into $\C^2$ by $\Cl_1 \otimes_{\Cl_1} \C^2 \cong \C^2$.
	\end{proof}

        On the above tensor product we may now introduce and analyse the product operator $\Dtimes$ of $S$ and $D_Y$. Later, this operator will be shown to be indeed an unbounded representative of the internal $KK$-product $i_! \otimes [Y]$.
	\begin{lemma}
		\label{lem:definition_product_operator}
Let 
		$$
		\Dtimes
  := S \otimes 1 + 1 \otimes_{\nabla^\E_{D_Y \otimes \sigma_2}} (D_Y \otimes \sigma_2)
		$$
                be the linear operator defined on the domain $\E^\infty \otimes_{C^\infty_0(Y) \otimes \Cl_1} (\Gamma^\infty(\Sigma_Y) \otimes \C^2)$. Then under the isomorphism $U$ from Lemma \ref{lem:product_unitary}
		$$
			U (\Dtimes) U^* = f \otimes \sigma_1 + i\gamma_s\left(D_{X_\bullet} + \frac{1}{2\Lambda}[D_{X_\bullet}, \Lambda] + \nabla^{\Sigma_Y}_{\partial_s}\right) \otimes \sigma_2
		$$
		where $f$ acts on $L^2(\pi^*\Sigma_X)$ as multiplication operator.
	\end{lemma}
	\begin{proof}
		We compute the $1 \otimes_{\nabla^\E} (D_Y \otimes \sigma_2)$ term first.
		For $\psi \otimes \phi \in \E^\infty \otimes_{C_0^\infty(Y) \otimes \Cl_1} (\Gamma^\infty(\Sigma_Y) \otimes \C^2)$ and writing $D$ for $D_Y \otimes \sigma_2$,
		\begin{align*}
			U(1 \otimes_{\nabla^\E_D} D)(\psi \otimes \phi) & = U\left(\gamma(\psi) \otimes D\phi + \nabla^\E_D(f)\phi\right), \\
			& = U\left(\gamma(\psi) \otimes D\phi + (1 \otimes [D, \psi] + \gamma(\psi) \otimes \frac{1}{2\Lambda}[D, \Lambda]) \phi\right), \\
			& = \gamma(\psi)D\phi + [D, \psi]\phi + \gamma(\psi) \frac{1}{2\Lambda}[D, \Lambda] \phi, \\
			& = D \psi \phi + \frac{1}{2\Lambda}[D, \Lambda] \psi \phi, \\
			& = (D + \frac{1}{2\Lambda}[D, \Lambda]) U(\psi \otimes \phi), \\
			& = (D_Y \otimes \sigma_2 + \frac{1}{2\Lambda}[D_Y, \Lambda] \otimes \sigma_2)U(\psi \otimes \phi).
		\end{align*}
		
		Now we can apply Lemma \ref{lem:Dirac_expansion} to replace $D_Y = i\gamma_s(D_{X_\bullet} - \frac{1}{2}\Tr(\II) + \nabla^{\Sigma_Y}_{\partial_s})$, we then obtain
		\begin{align*}
			U(1 \otimes_{\nabla^\E_D} D)U^* & = D_Y \otimes \sigma_2 + \frac{1}{2\Lambda}[D_Y, \Lambda] \otimes \sigma_2, \\
			& = i\gamma_s\left(D_{X_\bullet} - \frac{1}{2}\Tr(\II) + \nabla^{\Sigma_Y}_{\partial_s} + \frac{1}{2\Lambda}[D_{X_\bullet}, \Lambda] + \frac{1}{2\Lambda}[\nabla^{\Sigma_Y}_{\partial_s}, \Lambda]\right) \otimes \sigma_2, \\
			& = i\gamma_s\left(D_{X_\bullet}+ \nabla^{\Sigma_Y}_{\partial_s} + \frac{1}{2\Lambda}[D_{X_\bullet}, \Lambda]\right) \otimes \sigma_2
		\end{align*}
		using Lemma \ref{lem:Tr_II_as_derivative_Lambda}.
		
		The $U(S \otimes 1)U^*$ term is more straightforward,
		\begin{align*}
			U(S \otimes 1)(\psi \otimes \phi) & = U(S \psi \otimes \phi), \\
			& = U(f\sigma_1 \psi \otimes \phi), \\
			& = f\sigma_1 \psi \phi, \\
			& = (f\otimes\sigma_1) U(\psi \otimes \phi).
		\end{align*}
	\end{proof}
	
	\begin{lemma}
		\label{lem:unitary_twist}
		There is a unitary transformation $\Gamma(\Sigma_X) \otimes \C^2 \to \Gamma(\Sigma_X) \otimes \C^2$ such that
		$$
			\Dtimes \sim \left(D_{X_\bullet} + \frac{1}{2\Lambda}[D_{X_\bullet}, \Lambda]\right) \otimes 1 + \gamma_s  i\partial_s \otimes \sigma_1 - \gamma_s f \otimes \sigma_2
		$$
		and the grading $1 \otimes \sigma_3$ transforms to $\gamma_s \otimes \sigma_3$.
	\end{lemma}
	\begin{proof}
		Decomposing $\Sigma_X \otimes \C^2 = (\Sigma_X^+ \oplus \Sigma_X^-)\otimes\C^2$ by the grading $\gamma_s$, the unitary is given by
		$$
			U = \begin{pmatrix}
				i & 0 & 0 & 0 \\
				0 & 0 & 0 & i \\
				0 & 0 & 1 & 0 \\
				0 & -1 & 0 & 0
			\end{pmatrix}
		$$
		where $A \otimes B \leftrightarrow \begin{pmatrix}
			Ab_{11} & Ab_{12} \\ Ab_{21} & Ab_{22}
		\end{pmatrix}$.
		%The unitary is given by $\frac{1}{4}\left( (i+1) 1 \otimes 1 + (i-1) 1 \otimes \sigma_3 + (i+1) \gamma_s \otimes 1 + (i-1)\gamma_s \otimes \sigma_3 + (i-1) 1 \otimes \sigma_1 + (1-i) \gamma_s \otimes \sigma_1 + (i-1) 1 \otimes \sigma_2 + (1-i) \gamma_s \otimes \sigma_2 \right)$
		
		Existence of this unitary can also be established by verifying the algebraic relations among the Pauli matrices before and after this transformation.
	\end{proof}
	\begin{corollary}
		\label{cor:final_form_of_product}
		For $T = i\partial_s \sigma_1 - f \sigma_2$ there is a unitary isomorphism from $\E \otimes_{C_0(Y) \otimes \Cl_1}\left(L^2(\Sigma_Y) \otimes \C^2\right)$ to $L^2(\Sigma_X) \otimes L^2((-\epsilon, \epsilon), \C^2)$ such that 
		$$
			\Dtimes \sim \left(D_{X_\bullet} + \frac{1}{2\Lambda}[D_{X_\bullet}, \Lambda]\right) + \gamma_s \otimes T.
		        $$
                        on the respective domains $\E^\infty \otimes_{C^\infty_0(Y) \otimes \Cl_1}\left(\Gamma^\infty(\Sigma_Y) \otimes \C^2\right)$ and $\Gamma^\infty(\Sigma_X) \otimes C^\infty_0((-\epsilon, \epsilon), \C^2)$.
	\end{corollary}
	\begin{remark}
		\label{rem:definition_of_A}
		To keep future formulas more concise we will abbreviate
		\begin{equation*}
			A(x,s) := \frac{1}{2\Lambda(x, s)}\left[D_{X_s}, \Lambda|_{X_s}\right](x).
		\end{equation*}
		Note that $A$ is an odd endomorphism of $\pi^*\Sigma_X$, and as such acts on $L^2(\pi^*\Sigma_X)$.
		We will also write $A_s$ for the endomorphism of $\Sigma_{X_s}$ defined by $A_s(x) = A(x, s)$.
	\end{remark}
	
	It is this final unitary equivalence that we will leverage in section \ref{sec:analysis} to obtain the required analytical properties of the product cycle and then again in section \ref{sec:recovering} to recover $D_X$ from $\Dtimes$.

\section{Analysis of the unbounded product}
\label{sec:analysis}

	In this section we will further investigate the analytical properties of the product operator $\Dtimes$ defined in Lemma \ref{lem:definition_product_operator} through the unitary equivalent form found in Corollary \ref{cor:final_form_of_product}.
	We will show that $D_1 := D_{X_\bullet} + A$ and $D_2 := \gamma_s T$ form a weakly anti-commuting pair in the sense of \cite{LeschMesland_SelfadjointSums} so that $D_\times = D_1 + D_2$ is self-adjoint and that $D_1 + D_2$ has compact resolvent.
	
	\begin{definition}[{\cite[Definition 2.1]{LeschMesland_SelfadjointSums}}]
		\label{def:weakly_anticommuting}
		A pair of self-adjoint operators $(D_1, D_2)$ on a Hilbert space $H$ is {\em weakly anticommuting} if
		\begin{enumerate}
			\item There are constants $C_0, C_1, C_2$ such that for all $\psi \in \dom([D_1, D_2])$
			$$
				\| [D_1, D_2]_+ \psi \|^2 \leq C_0 \|\psi\|^2 + C_1 \| D_1 \psi \|^2 + C_2 \| D_2 \psi \|^2.
			$$
			\item There is a core $E \subset \dom(D_2)$ such that $(S+i\lambda)^{-1}(E) \subset \dom([D_1, D_2])$ for $\lambda \in \R$ large enough.
		\end{enumerate}
		Here $[D_1, D_2]_+ = D_1D_2 + D_2D_1$ is the anticommutator.
	\end{definition}
	\begin{remark}
		In \cite{LeschMesland_SelfadjointSums} this definition is given in the context of Hilbert modules. All their results we use continue to hold in that setting, but some of the discussion simplifies in the special case of Hilbert spaces.
	\end{remark}

	We first need to establish that $D_1$ and $D_2$ separately are self-adjoint on appropriate domains in $L^2(\Sigma_X) \otimes L^2((-\epsilon, \epsilon), \C^2)$.
	\begin{lemma}
		\label{lem:D_Xs_self_adjoint}
		The operator $D_{X_\bullet} + A$ with domain $H^1(\Sigma_X) \algotimes L^2((-\epsilon, \epsilon), \C^2) \subset L^2(\Sigma_X) \otimes L^2((-\epsilon, \epsilon), \C^2)$ is essentially self-adjoint.
	\end{lemma}
	\begin{proof}
		First we note that $D_{X_\bullet} + A$ is symmetric, each $D_{X_s}$ is symmetric relative to the volume form induced by $g_Y|_{X_s}$, the term $A$ from remark \ref{rem:definition_of_A} is exactly the correction required to make $D_{X_s}$ symmetric relative to the volume form induced by $g_Y|_{X_0}$.

		Consider the operator $D$ on the Hilbert $C_0\left((-\epsilon, \epsilon)\right)$-module $C_0\left((-\epsilon, \epsilon), L^2\left(\Sigma_X \otimes \C^2\right)\right)$ with domain $C_0\left((-\epsilon, \epsilon), H^1\left(\Sigma_X \otimes \C^2\right)\right)$, given by $D(f) = (D_{X_s}+A_s)f(s)$ for a continuous $H^1\left(\Sigma_X \otimes \C^2\right)$-valued function $f$.
		
		For each fixed $s$ the operator $D_{X_s}+A_s$ an elliptic symmetric first order differential operator and thus self-adjoint, in particular closed, on $H^1(\Sigma_X \otimes \C^2)$.
		A straightforward pointwise convergence argument then shows that $D$ is closed. 
		By the local-global principle \cite[Thm. 1.18]{Pierrot_LocalGlobal} self-adjointness of the localizations $D_{X_s}+A_s$ then implies that $D$ is self-adjoint and regular on $C_0\left((-\epsilon, \epsilon), H^1\left(\Sigma_X \otimes \C^2\right)\right)$.

		The essential self-adjointness of $D_{X_\bullet} + A$ on $H^1(\Sigma_X) \algotimes L^2\left((-\epsilon, \epsilon), \C^2\right)$ then follows by essential self-adjointness of $D \otimes 1$ on the internal tensor product of $C_0\left((-\epsilon, \epsilon), L^2(\Sigma_X \otimes \C^2)\right)$ with $L^2\left((-\epsilon, \epsilon)\right)$ considered as a left $C_0\left((-\epsilon, \epsilon)\right)$ module.
	\end{proof}
	
	\begin{lemma}
		\label{lem:T_self_adjoint}
		The operator $\gamma_sT_0:\Gamma_c^\infty(\pi^* \Sigma_X \otimes \C^2) \to L^2(\pi^*\Sigma_X \otimes \C^2)$ where $T_0 = i\partial_s \otimes \sigma_1 - f \otimes \sigma_2$ is essentially self-adjoint.
	\end{lemma}
	\begin{proof}
		In \cite[Proposition 2.10]{vanSuijlekomVerhoeven_AllSpheres} it was shown that $t_0=i\partial_s \sigma_1 - f \sigma_2:C_c^\infty((-\epsilon, \epsilon), \C^2) \to L^2((-\epsilon, \epsilon), \C^2)$ is essentially self-adjoint.
		The operator $\gamma_sT_0 = \gamma_s \otimes t_0$ on $\Gamma^\infty(\Sigma_X) \otimes C_c^\infty((-\epsilon, \epsilon), \C^2)$ is then essentially self-adjoint as well by \cite[Theorem VIII.33]{ReedSimonI}.
	\end{proof}
	\begin{remark}
		We will use the letter $T$ to denote both the self-adjoint operators $T = \overline{t_0}:\dom(T) \to L^2((-\epsilon, \epsilon), \C^2)$ and $T = \overline{T_0}:\dom(T) \to L^2(\pi^*\Sigma_X \otimes \C^2)$.
		Which operator is intended will be clear from context.
	\end{remark}
	
	From here on when we write $D_1$ and $D_2$ we mean the self-adjoint closures corresponding to the domains in Lemmas \ref{lem:D_Xs_self_adjoint} and \ref{lem:T_self_adjoint}.
	
	\begin{proposition}
		The pair $(D_1, D_2)$ is a pair of weakly anticommuting operators.
	\end{proposition}
	\begin{proof}
		We start by showing that the anticommutator $[D_1, D_2]_+$ is relatively bounded by $D_1$, establishing the first condition in Definition \ref{def:weakly_anticommuting} with $C_2 = 0$.
		
		As $\gamma_s D_1 = -\gamma_s D_1$ and the function $f$ in $T$ depends only on $s$, we find that
		\begin{align*}
			[D_1, D_2]_+ & = [D_1, \gamma_s T]_+, \\
			& = -\gamma_s [D_1, (i\partial_s \otimes \sigma_1  - f \otimes \sigma_2)], \\
			& = -\gamma_s [D_1, (i\partial_s \otimes \sigma_1)], \\
			& = \gamma_s (1 \otimes i\sigma_1)(\partial_s D_1).
		\end{align*}
				
		The $s$-derivative of $D_1$ is, for each fixed $s$, a first-order differential operator $H^1(\Sigma_X) \to L^2(\Sigma_X)$.
		As $D_1$ is, for each fixed $s$, an elliptic first-order differential operator, $(\partial_s D_1)(D_1 + i)^{-1}$ is a bounded operator on $L^2(\Sigma_X)$ by G{\aa}rding's inequality, say with bound $C_s$.

		Then $(\partial_s D_1)(D_1 + i)^{-1}:L^2(\pi^* \Sigma_X) \to L^2(\pi^* \Sigma_X)$ is bounded by $\sup_{s \in (-\epsilon,\epsilon)} C_s$.
		Since the coefficients of $D_1$ vary smoothly, we may choose $C_s$ to be continuous in $s$ for $s \in (-\epsilon_0, \epsilon_0)$ so that  $(\partial_s D_1)(D_1 + i)^{-1}$ is indeed bounded by some $C > 0$ for $s \in [-\epsilon, \epsilon]$, as $\epsilon < \epsilon_0$.
		Taking $C_0 = C_1 = C$ in Definition \ref{def:weakly_anticommuting} then shows that $(D_1, D_2)$ satisfies the first condition.
		
		For the second condition in Definition \ref{def:weakly_anticommuting}, note that $\Gamma^\infty(\Sigma_X) \otimes C_c^\infty((-\epsilon, \epsilon), \C^2)$ is a core for $D_2$ by Lemma \ref{lem:T_self_adjoint} and this space is contained in $\dom([D_1, D_2])$.
		As $D_1$ is a differential operator both it and its resolvents preserve the smooth functions, hence $\Gamma^\infty(\Sigma_X) \otimes C_c^\infty((-\epsilon, \epsilon), \C^2)$ is a suitable core.
	\end{proof}
	\begin{corollary}
		\label{cor:Dtimes_selfAdjoint}
		The operator $D_\times = D_1 + D_2$ is self-adjoint on $\dom(D_1) \cap \dom(D_2)$ and there exists a constant $C$ such that $\|\psi\|^2 + \|D_1 \psi\|^2 + \|D_2 \psi\|^2 \leq C\left(\|\psi\|^2 + \|(D_1+D_2)\psi\|^2\right)$.
	\end{corollary}
	\begin{proof}
		This follows immediately from \cite[Theorem 2.6]{LeschMesland_SelfadjointSums}.
	\end{proof}
	
	Besides essential self-adjointness, we need to check that $\Dtimes$ has compact resolvents.
	
	\begin{proposition}
	  The operator $\Dtimes \sim D_{X_\bullet} + A + \gamma_s T$ has compact resolvents.
          \label{prop:compact}
	\end{proposition}
	\begin{proof}
		By Corollary \ref{cor:Dtimes_selfAdjoint} convergence in the $\Dtimes$-norm implies convergence in both the $D_1$- and $D_2$-norms.
		Therefore $D_1 + D_2$ is essentially self-adjoint on $\Gamma^\infty(\pi^*\Sigma_X\otimes\C^2)$, as both $D_1$ and $D_2$ are essentially self-adjoint on that domain.
		So if $\psi \in \dom(D_1 + D_2)$ we can write $\psi$ as the limit of smooth functions in the $D_1$- and $D_2$-norm simultaneously.
		
		Since $D_1$ is elliptic along $X$, $D_1$-norm convergence implies the existence of weak derivatives along $X$ and similarly $D_2$-norm convergence implies the existence of weak $s$-derivatives \cite[Lemma 2.11]{vanSuijlekomVerhoeven_AllSpheres}.
		Together this implies that $\psi \in H^1(\pi^*\Sigma_X \otimes \C^2)$ which is compact by the Rellich embedding theorem.
		Hence the domain of $D_1 + D_2$ is compact in $L^2(\pi^*\Sigma_X \otimes \C^2)$ proving that $D_1 + \gamma D_2$ has compact resolvents.
	\end{proof}

	When combining the above Propositions \ref{cor:Dtimes_selfAdjoint} and \ref{prop:compact} we thus come the conclusion that the product operator $\Dtimes$ defines a spectral triple for $C(X)$:
	\begin{corollary}
		The triple $(C(X), \E \otimes_{C_0(Y) \otimes \Cl_1}  L^2(\Sigma_Y) \otimes  \C^2, \Dtimes)$ is a spectral triple. 
	\end{corollary}

\section{Recovering the embedded manifold from the product}
\label{sec:recovering}

	We will start this section by showing how the unbounded $KK$-cycle $ (L^2((-\epsilon, \epsilon), \C^2), T)$ describes the multiplicative unit in $KK_0(\C, \C)$.
	We then establish an {\em unbounded homotopy} between $\Dtimes$ and the external product $D_X \times T$ in the sense of \cite{vandenDungenMesland_Homotopy}. 
	This shows that our unbounded construction is indeed a refinement of the factorization in bounded $KK$-theory $[X]\times \1 = \imath_! \otimes [Y\otimes\C^2]$ \cite{ConnesSkandalis_LongitudinalIndex}.
	After this we will show how to obtain the spectral triple for $[X]$ from the family of $KK$-cycles $\{\Dtimes\}_{0 < \epsilon < \epsilon_0}$ and interpret the curvature in this light as well.

	\begin{lemma}
		\label{lem:T_represents_1}
		The unbounded KK-cycle from $\C$ to $\C$ given by 
		$$
			(L^2((-\epsilon, \epsilon), \C^2), T_\epsilon = i\partial_s \sigma_1 - f_\epsilon \sigma_2; \sigma_3)
		$$
		where $f(s) = -\frac{\epsilon}{2\pi}\tan\left(\frac{\epsilon s}{2\pi}\right)$, represents the multiplicative unit in $KK_0(\C,\C)$.
	\end{lemma}
	\begin{proof}
		This statement is proven in \cite[Section 2.4]{vanSuijlekomVerhoeven_AllSpheres}.
		In the interest of completeness but also brevity we will only recall the proof that $(L^2((-\epsilon, \epsilon), \C^2), T)$ represents the unit and omit the elementary but lengthy proof that $T$ is self-adjoint and has compact resolvent.
		
		As $KK_0(\C, \C) \cong \Z$ via the index map, we compute the index of $T$.
		$T_+:L^2(-\epsilon, \epsilon) \to L^2(-\epsilon, \epsilon)$ is given by $i\partial_s - if_\epsilon$.
		So a function $u \in \ker(T_+)$ if and only if it satisfies the ordinary differential equation (ODE) $u' = f_\epsilon u$.
		The solutions to this ODE are $u(s) = C\cos\left(\frac{\pi s}{2\epsilon}\right)$, $C \in \C$.
		These $u$ are in the domain of $T_+$ so $\dim(\ker(T_+)) = 1$.
		
		On the other hand, $T_-:L^2(-\epsilon, \epsilon) \to L^2(-\epsilon, \epsilon)$ is given by $i\partial_s + if_\epsilon$.
		In this case the solutions to the ODE, $u' = -f_\epsilon u$, are $C \cos\left(\frac{\pi s}{2\epsilon}\right)^{-1}$, which are not in $L^2(-\epsilon, \epsilon)$ for $C \neq 0$.
		Hence $\dim(\ker(T_-)) = 0$, so that $\Index(T) = 1$.
	\end{proof}

	We now turn to establishing an unbounded homotopy between $[X] \otimes \1$ and $\Dtimes$, following \cite{vandenDungenMesland_Homotopy}.
	An unbounded homotopy between $KK_0(A, B)$-cycles $(H_0, D_0)$, $(H_1, D_1)$ is a $\overline{UKK}_0(A, C([0,1], B)$-cycle $(H, D)$ such that $\ev_i((H, D)) = (H \otimes_{\ev_i} B, D \otimes 1) \cong_u (H_i, D_i)$.
	The group $\overline{UKK}$ is defined in \cite{vandenDungenMesland_Homotopy}, based on a slight generalization of unbounded Kasparov cycles.
	
	\begin{proposition}
		\label{prop:homotopy}
		Let $\tilde{\L} = C\left([0, 1], L^2(\pi_\epsilon^* \Sigma_X \otimes \C^2)\right)$ as a $C(X)$-$C([0,1])$-Hilbert module, and define $\tilde{\Dtimes}$ by
		$$
			(\tilde{\Dtimes} \psi)(x, s, t) = \left(D_{X_{st}} \psi( \bullet, s, t)\right)(x) + A(x, st)\psi(x, s, t) + \gamma_s T_\epsilon \psi(x, s, t).
		$$
		Then $(\tilde{\L}, \tilde{\Dtimes})$ is an unbounded homotopy between $(L^2(\Sigma_X), D_X) \otimes (L^2((-\epsilon, \epsilon), \C^2),D_X \times T )$ and $(\E \otimes_{C_0(Y)} L^2(\Sigma_Y \otimes \C^2), \Dtimes)$.
	\end{proposition}
	\begin{proof}
		By \cite[Lemma 1.15]{vandenDungenMesland_Homotopy} we can obtain self-adjointness and regularity for $\tilde{D}$ by establishing that each $\tilde{D}_t$ is self-adjoint, $t \mapsto \tilde{D}_t$ is strongly continuous and that there is a core $C \subset L^2(\pi_\epsilon^* \Sigma_X \otimes \C^2)$, independent of $t$, for each $\tilde{D}_t$.
		The trio of these statements all follow from the same considerations as in Lemma \ref{lem:D_Xs_self_adjoint}, with $C = C^1(\pi^*\Sigma_X \otimes \C^2)$.
		
		We then also need to establish that $C(X) \subset \overline{\Lip^0}(\tilde{D})$ to obtain an unbounded operator homotopy.
		As for each $t\in [0,1]$ the operator $(\tilde{D}+i)^{-1}_t = (\tilde{D}_t+i)^{-1}$ maps into $H^1(\pi^*\Sigma_X \otimes \C^2)$, the operator $(\tilde{D}+i)^{-1}$ is compact.
		Hence $\Lip^0(\tilde{D}) = \Lip(\tilde{D})$.
		Then we note that $C^1(X) \subset \Lip(\tilde{D})$, so that indeed $C(X) \subset \overline{\Lip}(\tilde{D})$ and we have a valid $UKK$-cycle, in fact even a valid classical unbounded $KK$-cycle.
		
		Finally we note that for $t = 1$ we obtain exactly $(L^2(\pi_\epsilon^* \Sigma_X \otimes \C^2), \Dtimes)$ and for $t = 0$ we obtain $(L^2(\pi_\epsilon^* \Sigma_X \otimes \C^2), D_{X_0} \otimes 1 + \gamma_s T)$ using $\Lambda(x, 0) \equiv 1$.
		Under the obvious unitary isomorphism $L^2(\pi_\epsilon^* \Sigma_X \otimes \C^2) = L^2(\Sigma_X) \otimes L^2((-\epsilon, \epsilon), \C^2)$ we obtain the exterior product $[X] \otimes \1$.
	\end{proof}
	\begin{corollary}
		The unbounded $KK$-cycle $(L^2(\pi^*\Sigma_X \otimes \C^2), \Dtimes)$ represents the Kasparov product of the shriek class $\imath_!$ and the fundamental class $[Y \otimes \C^2]$.
	\end{corollary}
	\begin{proof}
		As the bounded transforms of $\imath_!^\epsilon$ and $D_Y$ represent the shriek class of $\imath$ and the fundamental class $[Y]$ in $KK_0(C(X), C_0(Y))$, $KK_1(C_0(Y), \C)$ respectively, their (bounded) Kasparov product represents $[X] \in KK_0(C(X), \C)$.
		By the homotopy obtained in Proposition \ref{prop:homotopy} and the fact that unbounded and bounded cycles and homotopies induce the same groups \cite[Theorem B]{vandenDungenMesland_Homotopy} the triple $(L^2(\pi^*\Sigma_X \otimes \C^2), \Dtimes)$ represents $[X] \in KK_0(C(X), \C)$ as well. 
		Hence it represents the Kasparov product of $\imath_!^\epsilon$ and $[Y \otimes \C^2]$.
	\end{proof}
	
	The homotopy in Proposition \ref{prop:homotopy} geometrically corresponds to stretching the metric in the normal direction for a given $\epsilon$, so that the neighbourhood of $X \subset Y$ becomes closer and closer to the product $X \times (-\epsilon, \epsilon)$.
	In particular the homotopy at time $t$ corresponds to the product $\imath^\epsilon_!$ with $Y$ where the metric on the normal part of the decomposition $TY|_{\tilde{\imath}(X \times (-\epsilon_0, \epsilon_0))} = TX \oplus \R \partial_s$, is scaled by $\frac{1}{t}$.
	While this process does not lose any topological information, it does erase geometric information by scaling the second fundamental form $\II$ by $t$ as $t \to 0$.
	
	In order to preserve the geometric information we instead consider the family of unbounded $KK$-cycles as $\epsilon \to 0$ without stretching the metric.
	To compare the various operators $\Dtimes$ we use the following unitary transformation.
	
	\begin{lemma}
		The map
		\begin{align*}
			& R_\epsilon \colon L^2\left(\pi_\epsilon^* \Sigma_X \otimes \C^2\right) \to L^2\left(\pi_1^* \Sigma_X \otimes \C^2\right), \\
			& (R_\epsilon \psi)(x, s) = \sqrt{\epsilon} \psi(x, \epsilon s).
		\end{align*}
		is unitary.
	\end{lemma}
	\begin{proof}
		It is a straightforward check that
		\begin{align*}
			(R_\epsilon^{-1}\psi)(x, s) = \frac{1}{\sqrt{\epsilon}} \psi\left(x, \frac{1}{\epsilon}s\right),
		\end{align*}
		and that this is indeed the adjoint:
		\begin{align*}
			\langle R_\epsilon \psi, \phi \rangle_{1} & = \int_X \int_{-1}^{1} \langle (R_\epsilon \psi)(x, s), \phi(x, s) \rangle_\Sigma ds dx, \\
			& = \int_X \int_{-1}^1 \langle \psi(x, \epsilon s), \phi(x, s) \rangle_\Sigma \sqrt{\epsilon} ds dx, \\
			& = \int_X \int_{-\epsilon}^\epsilon \langle \psi(x, t), \phi\left(x, \frac{1}{\epsilon}t\right) \rangle_\Sigma \frac{1}{\sqrt{\epsilon}} dt dx, \\
			& = \langle \psi, R_\epsilon^{-1}\phi \rangle_{\epsilon}.
		\end{align*}
	\end{proof}
	
	The operators we are interested in, $T_\epsilon$ and $D_{X_\bullet}+A_\bullet$, transform as follows.
	
	\begin{lemma}
		\label{lem:Re_scaling_T}
		$R_\epsilon T_\epsilon R_\epsilon^* = \frac{1}{\epsilon}T_1$
	\end{lemma}
	\begin{proof}
		We compute
		\begin{align*}
			\left(R_\epsilon T_\epsilon R_\epsilon^* \psi\right)(x, s) & = \sqrt{\epsilon}\left(\left(i \partial_s \otimes \sigma_2 - f_\epsilon \otimes \sigma_1\right)\left(R_\epsilon^* \psi\right)\right)(x, \epsilon s), \\
			& = \sqrt{\epsilon}\left(i\sigma_2\left(R_\epsilon^* \psi\right)'(x, \epsilon s) + \frac{\pi}{2\epsilon}\tan\left(\frac{\pi \epsilon s}{2\epsilon}\right) \sigma_1 \left(R_\epsilon^* \psi\right)(x, \epsilon s)\right), \\
			& = \sqrt{\epsilon}\left(i\sigma_2 \frac{1}{\epsilon\cdot\sqrt{\epsilon}} \psi'(x, s) + \frac{\pi}{2\epsilon} \tan\left(\frac{\pi s}{2}\right) \frac{1}{\sqrt{\epsilon}}\sigma_1 \psi(x, s) \right), \\
			& = \frac{1}{\epsilon} T_1 \psi.
		\end{align*}
	\end{proof}
	
	\begin{lemma}
		\label{lem:Re_scaling_families}
		Let $\{B_t\}_{t \in (-\epsilon, \epsilon)}$ be a family of operators on a Hilbert space $H$ and $B_\bullet:L^2(\pi_\epsilon^* H) \to L^2(\pi_\epsilon^* H)$ defined by $(B_\bullet \psi)(t) = B_t \psi(t)$.
		Then $(R_\epsilon B_\bullet R_\epsilon^*\psi)(t) = B_{\epsilon t}\psi(t)$.
	\end{lemma}
	\begin{proof}
		We compute
		\begin{align*}
			(R_\epsilon B_\bullet R_\epsilon^*\psi)(t) & = \sqrt{\epsilon}(B_\bullet R_\epsilon^*\psi)(\epsilon t), \\
			& = \sqrt{\epsilon}B_{\epsilon t}(R_\epsilon^*\psi)(\epsilon t), \\
			& = B_{\epsilon t} \psi(t).
		\end{align*}
	\end{proof}
	
	\begin{corollary}
		\label{cor:Re_scaling_Dtimes}
		$R_\epsilon ( \Dtimes) R_\epsilon^* = D_{X_{\epsilon \cdot \bullet}} + A_{\epsilon \cdot \bullet} + \frac{1}{\epsilon} \gamma_s T_1$, acting on $L^2(\pi^*_1 \Sigma_X \otimes \C^2)$.
	\end{corollary}
	\begin{proof}
		We have $\Dtimes = D_{X_\bullet} + A_\bullet + \gamma_s T_\epsilon$ by Lemma \ref{lem:unitary_twist}.
		With proper consideration of domains, $D_{X_\bullet} + A_\bullet$ is of the form described in Lemma \ref{lem:Re_scaling_families}, so that $R_\epsilon (D_{X_\bullet} + A_{\bullet}) R_\epsilon^* = D_{X_{\epsilon\cdot\bullet}} + A_{\epsilon \cdot \bullet}$, while $\gamma_s$ commutes with $R_\epsilon$ and $T_\epsilon$ transforms as in Lemma \ref{lem:Re_scaling_T}.
	\end{proof}
	
	We now establish an asymptotic expansion $R_\epsilon( \Dtimes) R_\epsilon^* \sim \frac{1}{\epsilon}\gamma_s \otimes T_1 + D_{X_0} + \O(\epsilon)$ as $\epsilon \to 0$, in the following sense.

	\begin{proposition}
		We have for any $\psi \in C^1(\pi^*\Sigma_X \otimes \C^2)$, $\epsilon < \epsilon_0$
		$$
			\left\|\left(R_\epsilon( \Dtimes) R_\epsilon^* - \frac{1}{\epsilon}\gamma_sT_1 - D_{X_0}\right)\psi\right\| \leq C \epsilon (\| \psi \| + \| \psi \|_1)  
		$$
		for some $C > 0$ depending only on the metric of $Y$.
	\end{proposition}
	\begin{proof}
		By Corollary \ref{cor:Re_scaling_Dtimes} we need to show that
		$$
			\| (D_{X_{\epsilon \cdot \bullet}} + A_{\epsilon\cdot\bullet} - D_{X_0})\psi \|_{L^2} \leq C \epsilon (\| \psi \| + \| \psi \|_1).
		$$
		        As $A_0 = 0$ and $D_{X_{\epsilon \cdot 0}} = D_{X_0}$ this follows from the mean value theorem if we have an upper bound on $\|\partial_s(D_{X_s} + A_s)\psi\|$ which is uniform in $s \in (\epsilon, \epsilon)$. As the family $\{\partial_s(D_{X_s} + A_s)\}_{s \in (-\epsilon_0, \epsilon_0)}$ is a smooth family of first-order differential operators, with coefficients depending continuously on the metric of the submanifold $\tilde \imath(X \times (-\epsilon_0,\epsilon_0)) \subseteq Y$, the norm of $(\partial_s(D_{X_s} + A_s))\psi $ on the compact submanifold $\tilde{\imath}(X \times [-\epsilon, \epsilon])$ can be bounded by a multiple of the $C^0$ and $C^1$-norm of $\psi$. 
	\end{proof}
        In the spirit of renormalization methods in quantum field theory, we may consider the above subtraction of the term $\frac 1 \epsilon T$ from $\Dtimes$ as a renormalization prescription on
        how to obtain the relevant part $D_{X_0}$ for the fundamental class of the embedded manifold from the product operator.

	\begin{lemma}
		The operator $A_{\epsilon \cdot \bullet} \in \End(\pi^*_1 \Sigma_X)$ defined by
		$$
			A_{\epsilon \cdot \bullet}(x, s) = A_{\epsilon s}(x) = \frac{1}{2\Lambda(x, {\epsilon s})}\left[D_{X_{\epsilon s}}, \Lambda|_{X_{\epsilon s}}\right](x) \in \End(\Sigma_X)
		$$
		converges to 0 in norm as $\epsilon \to 0$.
	\end{lemma}
	\begin{proof}
		As $\Lambda$ is a smooth function, $A_{\epsilon\cdot\bullet}$ is a smooth family of bounded operators.
		We further have $\Lambda(x, 0) \equiv 1$, so that $A_0(x) \equiv 0$.
		Thus, by compactness of $X$, $A_{\epsilon\cdot\bullet}$ converges to 0 in norm.		
	\end{proof}
	\begin{corollary}
		The curvature of $\imath_!^\epsilon$ converges to $-\frac{1}{4}\Tr(\II)^2$ in norm as $\epsilon \to 0$.
		Note that $\frac{1}{\dim(X)}\Tr(\II)$ is the mean curvature.
	\end{corollary}
	
	Geometrically this process amounts to zooming in to a smaller and smaller neighbourhood of $X$.
	As opposed to the stretching seen in the homotopy in Proposition \ref{prop:homotopy} which erases the geometric information in the second fundamental form, this approach preserves this information while still recovering $X$ itself in limit.
	
\section{Summary and Outlook}
\label{sec:outlook}

	Given a compact Riemannian manifold $X$, a Riemannian $\spinc$ manifold $Y$ and an oriented codimension 1 Riemannian embedding $\imath:X \to Y$ we constructed a family of unbounded $KK$-cycles $\imath_!^\epsilon$ between $C(X)$ and $C_0(Y)$ such that they represent the shriek class $\imath_!$ at the level of $KK$-theory. Moreover, the unbounded product $\Dtimes$ represents the (bounded) internal Kasparov product $[X] = \imath_!^\epsilon \otimes [Y]$.
	In the limit $\epsilon \to 0$ we have recovered the Dirac operator $D_X$ as the constant term in the asymptotic expansion of the product operator $\Dtimes$ in $\epsilon$. This can be considered as the ``renormalized'' term after subtracting the ``divergent'' term proportional to the index class in a quantum field theory sense, with $\epsilon$ playing the role of a regulator. 
	We also recover the mean curvature of the embedding $\imath$ in the curvature of the $KK$-cycles $\imath_!^\epsilon$ in the $\epsilon \to 0$ limit.

	A large part of this analysis comes down to comparing two different metrics on $X \times (-\epsilon, \epsilon)$, one obtained as the pullback of the projection onto $X$ and the other obtained by pullback of the metric on $Y$ using Fermi coordinates.
	Constructing cycles that intertwine different metrics on the same space might be of independent interest and help in the analysis of higher codimension cases.

	Indeed, a natural next step is extending this construction to embeddings of higher codimension.
	While Fermi coordinates do exist in this situation they are no longer canonical and thus finding a factorization of the Dirac operator on $Y$ in terms of the Dirac operator on $X$ and a normal part will require more work. However, we do expect that the basic elements of the codimension 1 case will generalize.

	Another open question concerns the specifics of our construction of the $\imath_!^\epsilon$.
	For example the particular choice of $f(s) = -\frac{\pi}{2\epsilon} \tan \left( \frac{\pi s}{2\epsilon} \right)$.
	This choice of $f$ is made because it satisfies the differential equation $f^2-f'=-\left(\frac{\pi}{2\epsilon}\right)^2$ which we use to prove that the index-class $T$ is self-adjoint.
	In general we believe it suffices if $f^2-f'$ and $f^2+f'$ are bounded below.
	Indeed, in this case, it is straightforward to show that $T^2$ is bounded below so that $T^2$ at least has a self-adjoint Friedrichs extension.	

\newcommand{\noopsort}[1]{}\def\cprime{$'$}

\end{document}